\documentclass[12pt]{amsart}

\newtheorem{lemma}{Lemma}%[section]
\newtheorem{thm}[lemma]{Theorem}
\newtheorem{cor}[lemma]{Corollary}

\title{Bernoulli convolutions-2023}

\author{Nikita Sidorov}
\address{University of Great Northern Tower, 1 Watson Street, Manchester M3 4EE, United Kingdom}
\email{nikita.a.sidorov at gmail.com}
\dedicatory{To Erd\H os P\'al, the greatest mathematician of the last century}
\date{\today} 

\begin{document}

\begin{abstract}
Let $\theta\in(1,2)$, and $\mu_{\theta}$ be the 
Bernoulli convolution parametrized by $\theta$, that is, the measure
corresponding to the distribution of the random variable
$\sum_{n=1}^{\infty} a_n\theta^{-n}$, where the $a_n$ are i.i.d.
with probability of $a_n=0$ equal to $\frac12$.

As is well known, $\mu_\theta$ is either equivalent to
the Lebesgue measure on $\text{supp}(\mu_\theta)$,
or singular.

Recall that an algebraic integer $>1$ is called Pisot
if all its other Galois conjugates are smaller than 1 in modulus.
It is known that $\mu_\theta$ is singular with 
$\dim\mu_\theta<1$ if $\theta$ is Pisot. 

An algebraic integer $\theta$ greater than 1 is called a Salem number if
all its other Galois conjugates are of modulus 1, except
$\theta^{-1}$.

I shall prove that 

(1) $\dim\mu_\theta=1$ if $\theta$ is an algebraic non-Pisot number.

(2) if $\theta$ is Salem, then $\mu_\theta$
is equivalent to the Lebesgue measure on $\text{supp}(\mu_\theta)$,
with an unbounded density in $L^p(\text{supp}(\mu_\theta))$ for all $p<\infty$.

(3) Define
\[
\beta_{\theta,x,n}=\#\left\{a_1\dots a_n: \exists a_{n+1}\dots\text{such that\ } 
x=\sum_{k=1}^{\infty}a_n\theta^{-k}\right\}.
\]
Then 
\[
\lim_{n\to\infty}\sqrt[n]{\beta_{\theta,x,n}}=\theta^{\dim\mu_\theta}\text{\ for}\ \mu_\theta-\text{a.e.} x.
\]

(4) Put
\[
\bigcup_{n=1}^\infty\left\{\sum_{k=1}^{n}a_k\theta^k\mid a_k\in\{-1,0,1\}\right\}=
\{y_0(\theta)<y_1(\theta)<\cdots\},
\]
and
\[
\ell(\theta)=\liminf_{n\to\infty}(y_{n+1}(\theta)-y_n(\theta)).
\]
I shall present a short proof of De-Jun Feng's famous theorem stating that
$\ell(\theta)=0$ for all non-Pisot $\theta$.
\end{abstract}

\maketitle

If $\theta$ is transcendental, then $\dim\mu_\theta=1$ \cite{V}.
I shall assume henceforth $\theta$ to be an algebraic number.

Let $\theta_1=\theta,\dots,\theta_d$ be the Galois conjugates
of $\theta$ and $\mathcal M(\theta)
=\prod_{j: |\theta_j|>1} |\theta_j|$, the Mahler measure of $\theta$
if $\theta$ is an algebraic integer.

 Put
\[
D_n(\theta)=\left\{\sum_{k=1}^n a_k\theta^{k}\mid a_k\in\{0,1\}\right\},
\]
and $d_n(\theta)=\#D_n(\theta)$. Put 
\[
L(\theta)=\limsup_{n\to\infty}(y_{n+1}(\theta)-y_n(\theta)).
\]
Recall that $\theta$ is called {\sl Perron} if all other Galois
conjugates of $\theta$ are smaller than $\theta$ in modulus. 
By \cite[Theorem 2.1]{SS}, $L(\theta)=0$ if $\theta<\sqrt2$ and
$\theta$ is not Perron. 
\footnote{Condition (1) in \cite[Theorem~2.1]{SS} comes with an extra assumption that 
$-\theta$ is not a Galois conjugate of $\theta$ which means that
$\theta$ is not a root of a polynomial with no odd powers. I shall
deal with this case in the proof of Theorem~\ref{thm1}.} 
 This implies 
 \begin{equation}\label{0}
 d_n(\theta)\ge c_n(\theta)\cdot\theta^{n}\text{\ for some
unbounded\ } c_n(\theta)>0.
\end{equation} 
Let $\theta_2,\dots,\theta_{s+1}$ denote the non-real conjugates of $\theta=\theta_1$, with $s=0$ if there are none.

By the above and \cite[Lemma~1.51]{Ga}, there exists $C_\theta>0$ such that
for any interval $J$ with adjacent endpoints in $D_n(\theta)$,
\begin{equation}\label{1}
C(\theta)\cdot n^{-s}\mathcal M(\theta)^{-n}\le\mu_\theta(J)\le d_n(\theta)^{-1}\le c_n(\theta)^{-1}\theta^{-n}.
\end{equation}

\begin{thm}\label{thm1}
\begin{enumerate}
  \item $\dim\mu_\theta=1$ iff $\theta$ is not Pisot. 
  \item $\mu_\theta$ is equivalent to the Lebesgue measure on $\text{supp}\ \mu_\theta$
  if $\theta$ is Salem, with density in $L^p(\text{supp}(\mu_\theta))$ for all $p<\infty$.
\end{enumerate}
\end{thm}
\begin{proof} (1) It suffices to show that for any interval $J$ with adjacent
endpoints in $D_n(\theta)$,
\begin{equation}\label{ratio}
  \frac{\log\mu_\theta(J)}{\log|J|}\ge1-\varepsilon
\end{equation} 
for any $\varepsilon>0$. 

Let us perform some reductions.

We may assume $-\theta$ is not a Galois conjugate of $\theta$,
otherwise it is a root of a polynomial with no odd powers,
so we keep extracting square roots of $\theta$ until we get $\alpha$
that is rid of this constraint. Since $\theta$ is a positive power of $\alpha$, 
$\ \alpha$ being fully-dimensional implies
$\theta$ being such.

It suffices to prove the claim for $\theta<\sqrt2$, in
view of $\mu_{\theta}=\mu_{\theta^2}*\theta\mu_{\theta^2}$. 

Thus, if $\theta$ is not Perron, then (\ref{1}) holds.
What is left is to obtain a lower bound for $|J|$ from \cite[Lemma~1.51]{Ga},
namely, a constant times the inverse of the product of the moduli of the other Galois
conjugates of $\theta$. Therefore, $|J|\gg (\mathcal M(\theta)-\varepsilon))^{-n}$.
Since $|J|\ll\theta^{-n}$ for all $\theta$ \cite{Re}, this proves (\ref{ratio}).

Assume now $\theta$ to be Perron and neither Pisot nor Salem. Then $\mathcal M(\theta)>\theta$,
and the same argument applies.

Finally, assume that $\theta$ is not an algebraic integer. Then all sums in $D_n(\theta)$ are distinct, 
whence the Garsia entropy $H_\theta=\log_\theta 2>1$ so $\theta$ is fully-dimensional
by {\sc Michael Hochman}'s celebrated formula $\dim\mu_\theta=\min\{H_\theta,1\}$ \cite{H}.

It suffices to recall that $\dim\mu_\theta<1$ for a Pisot $\theta$ by \cite{Ga} and \cite{H}.

\medskip\noindent 
(2). Assume $\theta$ to be Salem. Here $\mathcal{M}(\theta)=\theta$, and $s=d-2$, and $d/2\in\mathbb Z$, so
  (\ref{1}) is
  \[
  n^{2-d}\theta^{-n}\ll\mu_\theta(J)\ll\theta^{-n}.
  \]
  I shall prove a stronger claim:
\begin{equation}\label{salem}
  \mu_\theta(J)\gg n^{1-d/2}\theta^{-n}.
\end{equation}
  We have
\begin{equation}\label{3}
\sum_{k=1}^{n} \text{Re}(\theta_j^k)=\sum_{k=1}^{n}T_j(\cos\text{arg}\theta_j) \asymp n,
\end{equation}
where $T_j(t)=\cos\arccos jt$ is the $j$th Chebysh\"ev polynomial. The lower bound comes from $\alpha$ being badly approximable by the Roth theorem.

If $p$ is a polynomial in $\mathbb{Z}[x]$ with its coefficients bounded by 1
such that $p(\theta)\neq0$, then $p(\theta)p(\theta_2)\ldots p(\theta_d)\in\mathbb{Z}\setminus\{0\}$, whence
\[
|p(\theta)|\ge \left|\prod_{j=2}^{d}p(\theta_j)\right|^{-1}\gg  n^{1-d/2}\theta^{n}.
\]
(\ref{salem}) is proved. Hence
  \begin{equation}\label{log}
  \frac{\mu_\theta(J)}{|J|}\asymp \log|J|,  
  \end{equation}
  whence $\mu_\theta$ is equivalent to the Lebesgue measure,
  with the density $\delta_\theta\in L^p(\text{supp}\ (\mu_\theta))$ for any finite $p$.
  \end{proof}
  
  \medskip\noindent
  {\bf Remark.}
  $\delta_\theta$ is 0 on half the points of $\bigcup_{n=1}^\infty D_n(\theta)$ and $\infty$ on
  the rest. Therefore, $\delta_\theta$ is unbounded and discontinuous everywhere. I leave the details
  to the interested reader.
  
  The set $\{\|\theta^n\|\}_{n\ge1}$ is known to be dense but not equidistributed in $[0,1]$.
  Theorem~\ref{thm1} gives a more detailed view.
    
As a consequence of \ref{3}, follows
    \begin{cor}
    For any algebraic integer $\theta$,
    \[
    \text{tr}(\theta^n)=\sum_{j:|\theta_j|>1} |\theta_j|^n + O(n^{s/2}),\quad n\to\infty.
    \]
    This is obvious for the hyperbolic $\theta$ (with $s=0$) but novel -- to my best knowledge 
    -- for the non-hyperbolic $\theta$.
    \end{cor}
It follows from (\ref{salem}) that $d_n(\theta)\gg n^{d/2-1}\theta^{n}$
for a Salem $\theta$.
An upper bound is not needed for Theorem~\ref{thm1}, yet I believe it may be
of interest. 
    
\medskip\noindent\textbf{Towards absolute continuity.}
\begin{itemize}
  \item Every absolutely continuous measure is fully-dimensional but not necessarily
the other way round. The difference is crucial: ``dropping the logs'' --
compare (\ref{ratio}) with (\ref{log}).
  
  \item (\ref{1}) is crude, yet sufficient for
our purposes. The general question is to study the
quantity $$d_\theta=\lim_{n\to\infty}\sqrt[n]{d_n(\theta)}.$$
(The limit exists because $d_{n+k}(\theta)\le d_n(\theta)d_k(\theta)$ for all $n,k\ge1$.)
I believe knowing the exact value of $d_\theta$ should build a such-needed bridge from
dimension one to absolute continuity of Bernoulli convolutions -- about which we
know very little, generic results notwithstanding.
 
Namely, we only know that $\mu_\theta$ is absolutely
continuous with respect -- and hence equivalent by
\cite{S} -- to the Lebesgue measure
on the support of $\mu_\theta$ if $\theta$ is a {\sl Garsia number}, that is,
an algebraic integer with all Galois conjugates of modulus greater than~1
and their product being $\pm2$ as well as some algebraic
$\theta$ {\sl very} close to 1 recently discovered by {\sc P\'eter Varj\'u} \cite{V2}.

Notice that in the former case, $\theta<d_\theta=2<\mathcal M(\theta)$ unless $\theta=\sqrt{2}$.
This gives us hope for the case of algebraic $\theta$. 

\item We know from (\ref{1}) only that $$\theta\le d_\theta\le\mathcal M(\theta),$$
which proves sufficient when $\theta$ is Salem since in this case,
$\mathcal M(\theta)=\theta$. 
\footnote{In fact, (\ref{1}) only works because
either $\mathcal M(\theta)>\theta$, or $\theta$ is Salem, in which case
$s>0$.}
\item If $\theta$ is superexponentially well approximable, {\sc Michael Hochman}'s techniques from
\cite{H} may help proving absolute continuity of $\mu_\theta$.
\end{itemize} 

\medskip\noindent
(3)
As a direct corollary of Theorem~\ref{thm1}, we obtain a result that significantly
extends \cite[Theorem~1.1]{FS}.

\begin{thm}
  \[
\lim_{n\to\infty}\sqrt[n]{\beta_{\theta,x,n}}=\theta
\]
if $\theta$ is not Pisot and $\theta^{\dim\mu_\theta}$ otherwise.
\end{thm}
(4) 
 \begin{thm}\cite{F}
   $\ell(\theta)>0$ iff $\theta$ is Pisot. 
 \end{thm} 
 \begin{proof}
   If $\theta$ is not height one, the claim follows from the pigeonhole principle. 

\medskip\noindent
Let now $\theta$ be height one -- and therefore, an algebraic unit. 
Then
$\ell(\theta)=0$ -- if 
\begin{itemize}
  \item $\theta<\sqrt2$;
  \item not Perron;
  \item $-\theta$ is not a Galois conjugate of $\theta$ which means that
$\theta$ is not a root of a polynomial with no odd powers;
\end{itemize}
I shall now remove these constraints and complete the proof. 

\medskip\noindent
If $\theta<\sqrt2$ and a root of a polynomial with no odd powers,
we keep extracting square roots of $\theta$ until we get $\nu$
that is rid of this constraint. Then $\ell(\nu)=0$ implies $\ell(\theta)=0$.

\medskip\noindent
If $\theta=\sqrt2$, then $\ell(\theta)=0$ by the pigeonhole
principle. 

\medskip\noindent
If $\theta>\sqrt2$, then $\ell(\sqrt\theta)=0$ implies
$\ell(\theta)=0$. 

\medskip\noindent
If $\theta<\sqrt2$ and a non-Pisot Perron,  
its Mahler measure is
$\ge \theta$ and $>\theta$ unless $\theta$ is Salem. It suffices to recall that
$d_n(\theta)\ge c_n \theta^n$ with an unbounded $c_n$
and apply the pigeonhole principle. 
 \end{proof}

\end{document}